\documentclass[12pt,a4paper]{amsart}
\usepackage{amssymb, amstext, amscd, amsmath, color}
\usepackage{graphicx}

\usepackage{url}

\usepackage{hhline}

\newtheorem{thm}{Theorem}[section]
\newtheorem{cor}[thm]{Corollary}

\newtheorem{rem}[thm]{Remark}

\begin{document}

\newcommand{\FFock}{\mathcal{F}}
\newcommand{\kil}{\mathsf{k}}
\newcommand{\Hil}{\mathsf{H}}
\newcommand{\hil}{\mathsf{h}}
\newcommand{\Kil}{\mathsf{K}}
\newcommand{\Real}{\mathbb{R}}
\newcommand{\Rplus}{\Real_+}

%

\newcommand{\bC}{{\mathbb{C}}}
\newcommand{\bD}{{\mathbb{D}}}
\newcommand{\bN}{{\mathbb{N}}}
\newcommand{\bQ}{{\mathbb{Q}}}
\newcommand{\bR}{{\mathbb{R}}}
\newcommand{\bT}{{\mathbb{T}}}
\newcommand{\bX}{{\mathbb{X}}}
\newcommand{\bZ}{{\mathbb{Z}}}
\newcommand{\bH}{{\mathbb{H}}}
\newcommand{\BH}{{\B(\H)}}
\newcommand{\bsl}{\setminus}
\newcommand{\ca}{\mathrm{C}^*}
\newcommand{\cstar}{\mathrm{C}^*}
\newcommand{\cenv}{\mathrm{C}^*_{\text{env}}}
\newcommand{\rip}{\rangle}
\newcommand{\ol}{\overline}
\newcommand{\td}{\widetilde}
\newcommand{\wh}{\widehat}
\newcommand{\sot}{\textsc{sot}}
\newcommand{\wot}{\textsc{wot}}
\newcommand{\wotclos}[1]{\ol{#1}^{\textsc{wot}}}
 \newcommand{\A}{{\mathcal{A}}}
 \newcommand{\B}{{\mathcal{B}}}
 \newcommand{\C}{{\mathcal{C}}}
 \newcommand{\D}{{\mathcal{D}}}
 \newcommand{\E}{{\mathcal{E}}}
 \newcommand{\F}{{\mathcal{F}}}
 \newcommand{\G}{{\mathcal{G}}}
\renewcommand{\H}{{\mathcal{H}}}
 \newcommand{\I}{{\mathcal{I}}}
 \newcommand{\J}{{\mathcal{J}}}
 \newcommand{\K}{{\mathcal{K}}}
\renewcommand{\L}{{\mathcal{L}}}
 \newcommand{\M}{{\mathcal{M}}}
 \newcommand{\N}{{\mathcal{N}}}
\renewcommand{\O}{{\mathcal{O}}}
\renewcommand{\P}{{\mathcal{P}}}
 \newcommand{\Q}{{\mathcal{Q}}}
 \newcommand{\R}{{\mathcal{R}}}
\renewcommand{\S}{{\mathcal{S}}}
 \newcommand{\T}{{\mathcal{T}}}
 \newcommand{\U}{{\mathcal{U}}}
 \newcommand{\V}{{\mathcal{V}}}
 \newcommand{\W}{{\mathcal{W}}}
 \newcommand{\X}{{\mathcal{X}}}
 \newcommand{\Y}{{\mathcal{Y}}}
 \newcommand{\Z}{{\mathcal{Z}}}

\newcommand{\sgn}{\operatorname{sgn}}
\newcommand{\rank}{\operatorname{rank}}

\newcommand{\Isom}{\operatorname{Isom}}

\newcommand{\qIsom}{\operatorname{q-Isom}}

\newcommand{\sIsom}{\operatorname{s-Isom}}

\newcommand{\sep}{\operatorname{sep}}




\title[Elementary proofs of Kempe universality]{Elementary proofs of Kempe universality}

\author{S.C. Power}
\address{Dept.\ Math.\ Stats.\\ Lancaster University\\
Lancaster LA1 4YF \\U.K.}
\email{s.power@lancaster.ac.uk}




\maketitle

\begin{abstract}
An elementary proof is given to show that a parametrised algebraic curve
in the plane may be traced out, in the sense of A. B. Kempe, by a finite pinned linkage. Additionally it is shown that any parametrised continuous curve $\gamma:[0,1]\to \bR^2$  may be traced out by an infinite linkage where the valencies of the joints is uniformly bounded. We also discuss related Kempe universality theorems and give a novel correction of Kempe's original argument.
\end{abstract}

\section{Introduction}

In a succinct three page paper A. B. Kempe \cite{kem} indicated  ``a General Method of describing Plane Curves of the $n^{th}$ degree  by Linkwork".
 A \emph{linkage} is another name for a finite bar-joint framework $(G, p)$ in the usual sense (\cite{asi-rot}, \cite {gra-ser-ser}), with links being bars,
and a linkwork is a linkage which has some joints pinned in order to remove isometric planar motion. Kempe gave a method to show that any finite algebraic curve in the plane could be ``described" by the free motion of a joint of a planar pinned linkage. Precise versions of this assertion have been referred to as Kempe universality for the curves or algebraic sets under consideration.

The following theorem for polynomial curves is stated in  Connelly and Demaine \cite{con-dem} who propose that this is a precise version of what Kempe was trying to claim.

\begin{thm}\label{t:king} 
Let $C$ be a set in the plane that is the polynomial image of a closed interval. Then there is a planar bar-joint framework $(G,p)$ with some joints pinned such the set of positions of a particular joint over the  continuous  motion of $(G,p)$ is equal to $C$. 
\end{thm}

In Section 2 we give an elementary proof of a strengthened form of Theorem \ref{t:king} in which there is control of the parametrisation of the curve. Moreover the proof extends readily to rational curves. We also give a new infinite  linkage construction which, roughly speaking, simulates uniform approximation, and we use this to show that any \emph{continuous} image of $[0,1]$ in the plane is the Kempe trace, so to speak, of a joint of an infinite bar-joint linkage whose joint valencies are uniformly bounded. 
This improves the main result given Owen and Power in \cite{owe-pow-kempe}.

It has long been known that Kempe's arguments are incomplete
since Kempe's component linkages involve parallelograms and contraparallelograms with undesired bifurcation motion. See for example Hopcroft et al \cite{hop-et-al}. On the other hand Abbott \cite{abb} has recently shown that Kempe's original argument can be completed by applying various corrective bracing measures for these components. See also the sketch account of this in Demaine and O'Rourke \cite{dem-oro}. For completeness we discuss this approach and some alternative corrections in Section 4. Moreover in the final Section \ref{ss:kempe} we indicate a novel minimal correction involving only simple parallelogram bracings and a judiscious choice for the origin.

Theorem \ref{t:king} was originally obtained  for complex polynomial curves by  King (\cite{kin}, Theorem 2) who attributes the result to Kapovich and Millson \cite{kap-mil} and to Thurston. See also Jordan and Steiner \cite{jor-ste}. As we discuss later, these authors also consider how general compact real algebraic varieties may be represented  by pinned planar linkages, either isometrically or up to some form of isomorphism. See also King \cite{kin2}, \cite{kin3}. The arguments in these discussions differ from Kempe's approach and are lengthy, even for planar algebraic sets. 

The author would like to thank Herman Servatius for discussions and for suggesting the simple prismatic form of the separation translator.

\section{The main result}\label{s:first}
In the proof of Theorem \ref{t:speedmatch} we use four types of building block linkages, namely, \emph{linearisers}, \emph{separation translators}, \emph{separation copiers},  and \emph{separation multipliers}. Each of these is a planar bar-joint framework with a subset of joints satisfying a geometric property.
The novelty of the proof is in the use of linearising linkages. These simplify the construction of the crucial multiplier linkage and they also play a simple constraining role to construct the curve tracing joint from joints that trace the coordinate projections.

We first define the functions that these building block components realise, followed by the proof of Theorem \ref{t:speedmatch}.  We then detail how these components are constructed.
\medskip

A \emph{linearising linkage} is a linkage  which contains a finite subset of joints which are constrained to be co-linear to a fixed line and are otherwise capable of independent motion within a fixed bound.  
\medskip

A \emph{separation copier} is a linkage with two pairs of joints $A, B$ and $A', B'$ whose separation distances must coincide but whose continuous motions are otherwise free within a fixed distance of $A$. In other words, given a distance $R$ there is a distance copier such that if the joint $A$ is fixed then the joints $B$ and $A'$ are independently continuously movable in the disc centered at $A$ with radius $R$. For a given such position of $B, A'$ the remaining joint $B'$ is freely movable subject to the separation equality constraint, $|A'B'|=|AB|$. We view the separation distance $|AB|$ as an input which ensures the same separation distance for $A'$ and $B'$. 
\medskip

A \emph{separation translator} is a separation copier having the additional property  that the vectors $AB$ and $A'B'$ are parallel. In fact a general separation copier may be simply realised as the concatenation of a separation translator and a \emph{separation rotator} which is defined below.
\medskip

Finally, a \emph{separation multiplier} linkage has $3$ distinguished
joints $A,B,C$, with $A$ and $C$ constrained to the nonnegative $x$-axis and $B$ constrained to the nonnegative $y$-axis. The separations $|OA|$ and $|OB|$ are inputs that can vary independently within some bound and the separation $|OC|$ is the product of   $|OA|$ and $|OB|$. The constraining here is effected by linearisers.
\medskip

With these components we also define a \emph{separation powering linkage}, in which the input separation $|AB|$ determines the output separation $|A'B'|= |AB|^k$. This may be obtained by a concatenation of separation multipliers with connections by separation translators.
Also a \emph{ positive scalar multiplying linkage} is realizable from a separation multiplier linkage 
by adding a bar of that scalar length to one of the inputs.

We assume, as is usual, that bars and joints may  intersect so that the only constraint to the continuous motion of the unpinned joints is the  preservation of bar lengths.


\bigskip

\begin{thm}\label{t:speedmatch}
Let $C = \gamma([0,1])$ where $\gamma:[0,1]\to \bR^2$ is a polynomial curve. 
Then there is a finite planar bar-joint framework $(G,p)$,  with joints $p_1=(0,0), p_2=(1,0), p_L=(1,0), p_C= \gamma(0)$, with the following property.
For any continuous motion  $q(t), t\in [0,1]$, of $(G,p)$ with $q_i(t)=p_i$, for $i=1,2$, and $q_L(t) =(1+s(t),0)$ with
$s(t)\in [0,1]$ for all $t$, it follows that $q_C(t) = \gamma(s(t))$. Also there exists such a motion with $s(t)=t$ for all $t$.
\end{thm}

\begin{proof}Let $\gamma(t) = (x(t), y(t))$ and suppose that  $x(t)=a_0+a_1t+\dots +a_nt^n$. We may assume that $a_0=1$. The desired framework $\G= (G,p)$ is constructed around a triangular framework $\G_0$ with  joints $p_1=(0,0), p_2=(1,0), p_3=(0,1)$ and the $3$ connecting bars between them. 

Define first a  framework $\G_1$  by adding the joint $p_L=(1,0)$ and a lineariser for the triple to $p_1, p_2$ and $p_L$ so that the (pinned framework) motion of $p_L$ is constrained to 
linear motion on the line segment between $(1,0)$ and $(2,0)$.
The  joints $p_2$ and $p_L$ are regarded 
as an input pair with separation $t=|p_L(t)-p_2|$.

By the discussion above, for each $k=1,\dots, n$ with $a_k\neq 0$ there is a linkage $\H_k$ with input pair $p_2, p_L(t)$ and an output pair with separation $|a_k|t^k$. Each is obtained from a concatenation of  a powering linkage and a scalar multiplication linkage.  Moreover, for each $k$ we may constrain the output joints to the $x$-axis  by adding linearisers. 
Identify an output joint of $\H_k$ with an output joint of $\H_{k+1}$  in the following manner. We may assume that the coefficients $a_k$ are nonzero, the general case being similar. 
\medskip

(i) The output pair for $\H_1$ has one joint fixed at $(1,0)$ and the other output joint has motion to the right (resp. left) of $(1,0)$ if the sign of $a_1$ is positive (resp. negative).

(ii) For each $k=1,2,\dots ,n-1$, one output joint of $\H_{k+1}$ is identified with the free output joint of $\H_k$ while the other output joint has relative motion to the right (resp. left) if the sign of $a_{k+1}$ is positive (resp. negative).
\medskip

It follows that the unshared  joint of the final output pair is located at $(x(t),0)$ when $p_L$ is located at $(1+t,0)$.
\begin{center}
\begin{figure}[ht]
\centering
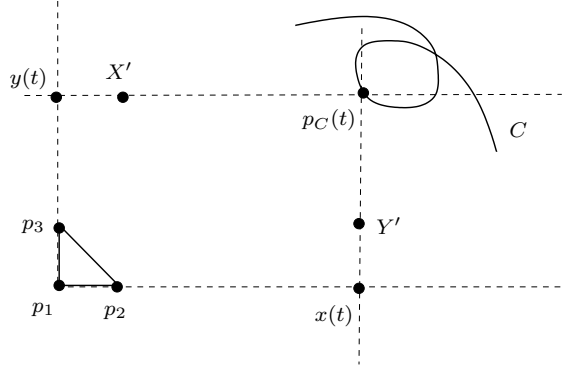
\caption{The determination of $p_C(t) =(x(t),y(t))$.}
\label{f:polydraw2}
\end{figure}
\end{center}
In the same way add further bars and joints so that the input separation  $t = |p_L(t)-p_2|$ determines an output joint at the position $(0,y(t))$. Finally,  add  separation translators and linearisers so that the output joints $(x(t),0)$ and $(0,y(t))$ determine a joint $p_C(t)$ at position $(x(t), y(t))$. This is indicated in Figure \ref{f:polydraw2}. The joint $X'= (1,y(t))$ (resp. $Y'=(x(t),1)$) is determined by a separation translator (not shown) to the fixed input pair $p_1, p_2$ (resp. $p_1, p_3$). Two  linearisers (not shown) maintain the colinearity of $p_C(t)$ with the pair $(0,y(t)), X'$ and the pair $(x(t),0), Y'$. The resulting linkage has the desired properties.
\end{proof}

Note that by swapping the output for an input the separation multiplier linkage becomes a \emph{separation divider linkage}. It follows  that one may prove the rational curve version of this theorem in the same manner. The methods also extends to polynomial and rational curves in higher dimensions. 

\begin{rem}
For a contrasting approach we note that Gallet et al \cite{gal-et-al} develop interesting  factorisation methods in noncommutative algebra for the construction of explicit linkages and in this way obtain a proof of Kempe universality for \emph{rationally parametrised curves}. See also Li et al \cite{li-et-al}. Part of the interest here is to generate rational curves efficiently and succinctly.

We also remark that algorithmic aspects of Kempe universality are also of interest in computational geometry and applied mathematics. See, for example, Abbott \cite{abb}, Gallet et al \cite{gal-et-al}, Gao et al \cite{gao-et-al}, Kobel \cite{kob} and Li et al \cite{li-et-al}. 
\end{rem}

\subsection{Linearising linkages}
Figure \ref{f:peaucellierB} shows  the famous
\emph{Peaucellier linkage}. Suppose that the joints at $p_1$ and $p_2$ are pinned. 
Then the geometry ensures that between two extremal positions the joint $q_1$ must move on a vertical line segment which is symmetric about a line through $p_1p_2$. At the extremal positions the linear motion breaks down. By adding a further "tethering joint" and bars connecting it to $p_3$ and $q_1$ one may obtain a \emph{strict} Peaucellier linkage in which the extremal positions are not achieved. In this case the only bifurcation in a continuous motion is the evident bifurcation for the tethering joint.


By replicating the mobile  structure attached to $p_1$ and $p_2$, with any finite multiplicity $n$, we obtain an \emph{$n$-fold Peaucellier  linkage}. This has $n$ joints, $q_1, \dots ,q_n$, whose relative motions, away from the extremal positions, are constrained to be co-linear. In this way we obtain a \emph{linearising linkage}.



\begin{center}
\begin{figure}[ht]
\centering
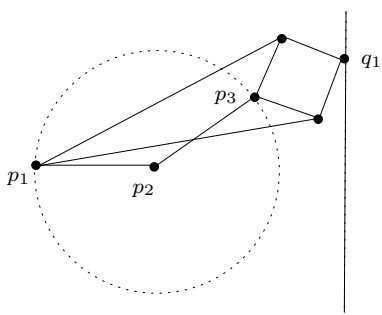
\caption{The Peaucellier linkage.}
\label{f:peaucellierB}
\end{figure}
\end{center}

\subsection{The separation multiplier}
The schematic diagram in 
Figure \ref{f:multiplier} underlies the construction of a separation multiplier linkage. The joints at the origin $O$ and at $X=(1,0), Y=(0,1)$ are connected by bars. Three linearising linkages (not shown)  constrain each of the set of joints  
 $\{O,X,A,C\}, \{O,Y,B\}$ and $\{B,A',C\}$ to be colinear.
Finally, the joint at $A'$ is
defined by a \emph{separation translator}, as constructed below, with input pair $\{Y, A\}$ and output pair $\{B, A'\}$. The resulting  geometry ensures that  
$|OC|=|OA||OB|$ for all input values  $ 0 \leq |OA|, |OB|\leq R$, for some $R>0$. 

\begin{center}
\begin{figure}[ht]
\centering
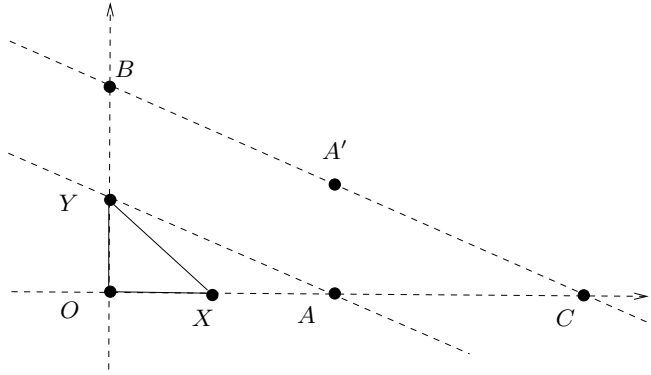
\caption{Construction scheme for a separation multiplier: $|OC|=|OA||OB|$.}
\label{f:multiplier}
\end{figure}
\end{center}

The parallel line principle here also underlies  Jordan and Steiner's  more subtle multiplicator linkage \cite{jor-ste}.

\subsection{Separation translator}
A separation translator, of the type we have defined above, may be constructed as indicated in Figure \ref{f:copierC}. Each of the $4$ parallelograms of the left hand bar-joint framework is to be prismatically braced, as indicated in the right hand figure, in order to avoid unwanted bifurcations when the parallelograms assume flat positions. The equality of $|BX|$ and $|AX|$ ensures that the separation distance between $A$ and $B$  ranges in the entire interval $0 \leq |AB|\leq 2|AX|$.

\begin{center}
\begin{figure}[ht]
\centering
\includegraphics[width=9cm]{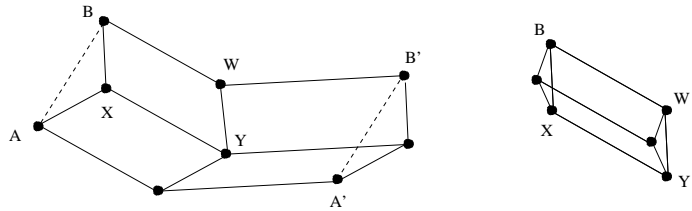}
\caption{The separation translator.}
\label{f:copierC}
\end{figure}
\end{center}
\begin{center}
\begin{figure}[ht]
\centering
\includegraphics[width=4cm]{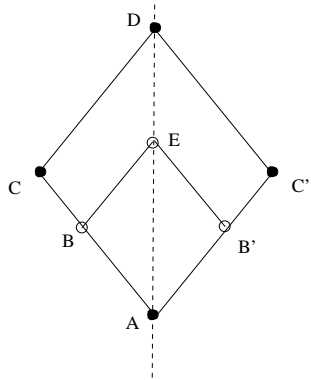}
\caption{Construction scheme for  a separation rotator:  
$|AB|=|AB'|$.}
\label{f:rotator}
\end{figure}
\end{center}

\subsection{Separation rotator}
Figure \ref{f:rotator} is a schematic diagram for the construction of a \emph{separation rotator}. It includes a $4$-bar parallelogram linkage with joints $A,C,D,C'$, whose $4$ bars have equal length. This parallelogram is  prismatically braced, as in the separation translator.
 The joint $E$ joins the equal length bars $BE$ and $B'E$, and there are three $3$-fold linearising linkages (not shown) which constrain $B$ (resp. $E$, resp. $B'$) to be colinear with
$AC$ (resp. $AD$ resp. $AC'$).

\section{Continuous curves}

In Owen and Power  \cite{owe-pow-kempe} 
we made use of Kempe universality for polynomial curves to construct an infinite pinned linkage to realise any given continuous curve $f : [0, 1] \to  \bR^2$ as the motion of one of its joints. 
Moreover the curve tracing is achieved strictly in the sense that the  motion of the joint is uniquely determined by an input parametrisation.

The structure of the proof in \cite{owe-pow-kempe} is as follows. 
One starts with a sequence of polynomially parametrised curves $t \to q_k(t)$, for $k=1,2, \dots$, which  approximate the curve $t \to f(t)$ uniformly.
An  increasing sequence of finite frameworks is assembled, each with the same two pinned joints,  such that the $k^{th}$ addition contains a joint $q_k$ which traces the curve
$t \to q_k(t)$.
It is then possible to add additional  bars to  tether a  joint $p_C$ to every $q_k$   so that  at time $t$ its position $p_C(t)$ is equal to the limit $\lim_k q_k(t)$. This infinite tethering contruction implies that the infinite framework has joints with infinite degree, that is, which are incident to infinitely many bars.
We also remark that a
parameterised form of Theorem \ref{t:king} was used without proof in \cite{owe-pow-kempe} and this incompleteness is  remedied by Theorem \ref{t:speedmatch}.

In the next theorem we obtain a stronger result wherein the continuous curve tracing is effected by a countably infinite bar-joint framework which has a uniform bound on the degrees of the joints. In the proof the infinite tethering is done indirectly by sequences of constraints to two virtual lines through $p_C$ which are parallel to the coordinate axes, and these virtual lines are created by sequences of linearisers.

\begin{thm}\label{t:kempeContsFiniteDegree}
Let $\gamma : [0, 1] \to  \bR^2$ be a continuous function.
Then there is a countable planar bar-joint framework $(G,p)$, of finite valency, with joints $p_1=(0,0), p_2=(1,0), p_L=(1,0), p_C= \gamma(0)$ such that the following holds.
For any continuous motion  $q(t), t\in [0,1]$, of $(G,p)$ with $q_i(t)=p_i$, for $i=1,2$, and $q_L(t) =(1+s(t),0)$ with
$s(t)\in [0,1]$ for all $t$, it follows that $q_C(t) = \gamma(s(t))$. Also there exists such a motion with $s(t)=t$ for all $t$.
\end{thm}

\begin{proof}
Without loss of generality we may assume that the range of $\gamma$ lies in the open unit square $(0,1)\times (0,-1)$. Consider first an $L$-shaped rigid grid framework $\G_1$ whose joints are located at $(k,0), (k,-1)$ and
$(0,-(k+2)), (1, -(k+2))$, for $k=0,1,2, \dots $. The bars are given by the horizontal and vertical unit length bars between these joints (shown in Figure \ref{f:biggrid}) together with  diagonal bars in each of the unit square subframeworks (not shown).

Using countably many  $3$-fold linearisers (no linearisers are shown) we can readily arrange that there are joints $q_0, q_1, q_2, \dots $, which lie on the same vertical line through $p_C=\gamma(0)$ and, respectively, lie on the horizontal lines with $y$ intercept $0, -1, -2, \dots $. We similarly constrain the sequence of joints $r_0, r_1, r_2, \dots,$ as shown.
Also, a lineariser is added to constrain the joint $p_2$ to the line $y=-1$, and sequences of linearisers are added to constrain joints $p_{2,k}$  to lie on $y=-k$, for $k=2,3,\dots ,$, and to have the same $x$-coordinate as $p_2$. (These joints, which lie vertically below $p_2$, are not shown. This completes the constructions of $\G_2 \supset \G_1$. 
Note that in any motion  which fixes the joint at the origin and $ p_1$ the joints $p_{2,k}$ move horizontally in unison with $p_2$.

\begin{center}
\begin{figure}[ht]
\centering
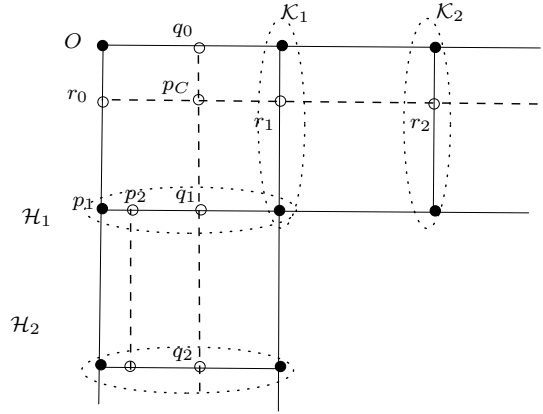
\caption{Construction scheme for $p_C$.}
\label{f:biggrid}
\end{figure}
\end{center}

For the next phase of the construction we consider the continuous coordinate functions $x(t)$ and $ y(t)$ of $\gamma(t)$.

By the uniform density of polynomial functions there is a sequence $x_1(t), x_2(t), \dots $ of polynomials such that for all $t$ in $[0,1]$ we have
$|x(t)-x_k(t)|\leq 1/k$ for all $t, k$.
By Theorem \ref{t:speedmatch} for each $k$ there is a finite framework $\H_k$ with input joints  $(0,-k)$ and  $p_{2,k}$, and with output joint $q_{C,k}$ (not shown in Figure \ref{f:biggrid}) at the initial position $(x_k(0),-k)$,  such that there is a unique continuous motion of $\H_k$ 
such that for all $t\in [0,1]$,
\[
p_{2,k}(t) = p_{2,k}+(t,0), \quad q_{C,k}(t) = (x_k(t),-k) 
\] 


The construction of the framework $\H_k$ is completed by adding tethering joints with bars  of length $1/k$ to the joints $q_k$ and $q_{C,k}$. Figure \ref{f:tether} indicates this tethering for $k=1$. These additions are chosen to guarantee that in any continuous motion of the union of $\G_1$ and $\H_k$ we have the inequalities $|q_k(t)-q_{C,k}(t)|\leq 1/k$.

Since all the joints $q_k$ move horizontally in unison, with the same $x$-coordinate, and since  $1/k$ tends to zero as $k$ tends to infinity, it follows that this $x$-coordinate for time $t$ must coincide with $x(t)$.
\begin{center}
\begin{figure}[ht]
\centering
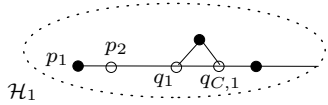
\caption{Tethering $q_1(t)$ to $q_{C,1}(t)$.}
\label{f:tether}
\end{figure}
\end{center}

In a similar way the graphs $\K_1, \K_2, \dots $ are constructed for a sequence of polynomial approximants to $y(t)$. The framework $(G,p)$ is defined to be the union of $\G_1$ and the frameworks $\H_k, \K_k$ for all $k$, together with linearisers for the triples $q_0, q_1, p_C$ and $r_0, r_1, p_C$ which constrain a joint $p_C$ to the position $(x(t), y(t))$ at time $t$, as required.
\end{proof}



Kempe's 1877 book, "How to draw a straight line", with its somewhat whimsical introduction, celebrates all manner of linkages and especially those that convert circular motion (a precise ruler and compasses construction) to the tracing of a straight line. Of the latter he asserts that:
\emph{If we are to draw
a straight line with a ruler, the ruler must itself have a
straight edge ; and how are we going to make the edge
straight ? We come back to our starting-point."} 

Bearing in mind the parameter control in the argument above and also the existence of continuous area filling curves, we obtain the following corollary which, in the same spirit, shows "How to draw a square".


\begin{cor} 
Let $R$ be the solid square
$[-\pi/4, \pi/4]\times [-\pi/4, \pi/4]$, let $S$ be its square boundary curve, and
let $C$ be the circle of unit radius.
\medskip

(i) There is a bounded countable pinned planar bar-joint framework with bounded valence which has joints $p_C, p_S$ whose continuous traces are $C$ and $S$, and where these sets are traced with equal speed.

(ii) There is a bounded countable pinned planar bar-joint framework
with bounded valence which has joints  $p_C, p_R$ whose continuous traces are $C$ and $R$.

\end{cor}

\section{Algebraic curves.}\label{s:kempe}

The \emph{continuous trace} of a joint in a pinned linkage is defined  to be the set of positions it may occupy in any  continuous  motion. In particular, such a set is pathwise connected. Let $f(x,y)$ be a real polynomial with zero set $V(f)$ in $\bR^2$.

\begin{thm}\label{t:kempeuniversality}
Let $C$ be a nonempty pathwise-connected component of 
$V(f)\cap D$ where $D$ is a closed disc. Then there exists a pinned planar bar-joint framework $(G,p)$ such that the continuous trace of one of the joints is equal to $C$.
\end{thm}

\begin{cor}\label{t:cor}
A compact connected real algebraic curve is the continuous trace of a single joint of a pinned linkage. 
\end{cor}

Abbott \cite{abb} provided proofs for these results by augmenting Kempe's original linkage construction to ensure that the parallelogram and contraparallelogram components have unique motions. In this section we give a similar completion and discuss some related results. 

The variety  $V(f)$  may contain branching points and so $C$ need not be equal to the image of a closed interval under a polynomial (or rational) function. 
This second branching issue was not considered by Kempe but is accommodated for in the hypotheses here.

Note that neither one of Theorem \ref{t:king} and   Theorem \ref{t:kempeuniversality}  subsumes the other, even for simple curves.
Indeed, the closed simple  curve satisfying the equation $x^4+y^4-1=0$  is  not polynomially (or rationally) parametrisable and yet Theorem \ref{t:kempeuniversality} provides a continuous trace linkage for it. On the other hand, if $C$ is a simple curve subset of the 
polynomially parametrised  curve  $t \to (3(3-t^2), t(3-t^2)), t\in \bR,$ which contains the crossover point $(0,0)$, as a non-endpoint, then $C$ does not satisfy the requirements of Theorem \ref{t:kempeuniversality}.

\subsection{Angle adder} We make use of a simple \emph{angle adding linkage}. In Figure
\ref{f:angleadder} the equality $|AB|=|A'B'|$ is forced by a separation copier so that the output bar for $OB'$ is at angle $\theta +\phi$. This is the sum of two input angles, for $OB$ and $OA'$, where the joints $A,B,A',B'$ are equidistant from $O$. Kempe indicates angle adders of essentially this design in the penultimate page of his book \cite{kem-book}.

\begin{center}
\begin{figure}[ht]
\centering
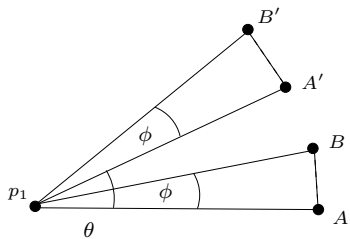
\caption{Angle adding linkage.}
\label{f:angleadder}
\end{figure}
\end{center}


\begin{proof}[The proof of Theorem \ref{t:kempeuniversality}]
Let
\[
f(x,y)= \sum_{0\leq i+j\leq n} a_{ij}x^iy^j
\] 
We may assume that $C$ lies in the first quadrant and that there are positive integers $a, b$ such that each point $(x, y)=\gamma(t)$ on $C$ is given by a unique pair $\theta, \phi$ with
\[
x= a\cos \theta +b\cos \phi,\quad  y = a\sin \theta +b\sin \phi
\]

Substituting and using standard trigonometric formulae obtain a representation
\begin{align*}
F(\theta,\phi) & = f(a\cos \theta +b\cos \phi, a\sin \theta +b\sin \phi)\\
 & = A_{0,0}+\sum_{0< r+s\leq n}A_{r,s}\cos(r\theta+\sigma_{r,s}s\phi + \psi_{r,s}) 
\end{align*}
where $\sigma_{r,s}\in \{-1,1\}, \psi_{r,s}\in \{0,\pi/2\}$.

Consider now any total ordering  of the nonzero coefficients $A_{r,s}$ in the sum which we may assume are positive by adjusting $\psi_{r,s}$ appropriately.  To construct the desired bar-joint framework $(G,p)$ fix $\theta, \phi$ for some point $(x,y)$ on the set $C$ and  write $p_C = p_C(\theta, \phi)$ for a joint at this location.
Let $\G_0$ be the corresponding $4$-bar parallelogram bar-joint framework, as in Figure \ref{f:kempe1}. The linkage construction (this is Kempe's main idea) is an inverse construction in which we add building block linkages, guided by the terms of $F(\theta, \phi)$, so that as $p_C(\theta, \phi)$ moves on the component $C$ a specific output joint is constrained to the line $x=-A_{0,0}$.

For the first nonzero coefficient $A_{r,s}$ in the ordering 
use an angle adding linkage to construct a 
linkage $\G_1'$ which includes 


(i) the pinned joints $p_1=(0,0), p_2=(1,0)$, 

(ii) the two bars of $\G_0$ that are incident to $p_1$, as angle inputs,

(iii) an output bar connecting $p_1$ and $p_3$, where $p_3$ has polar coordinates
\[
(A_{r,s}, r\theta+\sigma_{r,s}s\phi + \psi_{r,s})
\]


\begin{center}
\begin{figure}[ht]
\centering
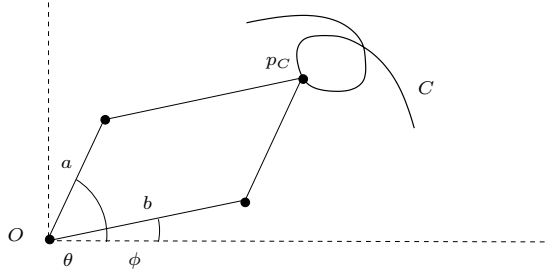
\caption{Changing to $\theta, \phi$ coordinates.}
\label{f:kempe1}
\end{figure}
\end{center}

Let $\G_1$ be the union of $\G_0$ and  $\G_1'$. 

Consider the next term with nonzero coefficient $A_{r,s}$.  Repeat the construction of $\G_1'$  to similarly create a joint $p_4'$ with the appropriate polar coordinates. Moreover using  separation translators create a joint $p_4$ such that
$p_4-p_3=p_4'-p_1$.

Continue in this manner to obtain, finally, a composite linkage $\G$ which includes a finite sequence of joints $p_3, p_4, \dots $ terminating in a joint, $p_L= p_L(\theta, \phi)$ say. By the construction the $x$-coordinate of this joint is equal to
\[
\sum_{0< r+s\leq n}A_{r,s}\cos(r\theta+\sigma_{r,s}s\phi + \psi_{r,s})
\]
It follows that if  $p_C$ moves continuously on the set $C$ then $p_L$ moves on a closed line segment of the line $x=-A_{0,0}$.

To complete the construction of $(G,p)$  add bars and joints to constrain $p_L$ to the line segment and to constrain $p_C$ to the closed disc. The first constraint may be achieved by 
a lineariser construction together with a tethering $2$ bar linkage connection to a fixed joint at the midpoint of the line segment. (Fixed joints may be constructed by connections to the pinned joints $p_1, p_2$.) The second constraint
may be achieved by adding a tethering $2$ bar linkage to constrain the separation between $p_C$ and  a fixed joint at the centre of the disc.

It remains to observe that for $(G,p)$, with its fixed joints, the continuous trace of $p_C$ is equal to $C$. This follows from the construction and the maximal connectedness hypothesis for $C$. These imply that in any continuous motion, $p_C(\theta',\phi')$ lies on $C$ if and only if $p_L(\theta', \phi')$
lies on $L$.
\end{proof}

Kapovich and Millson in fact consider the entire configuration space of a pinned linkage in which case the set of corresponding positions, or \emph{orbit}, of a particular joint may be a  disconnected  set. In other words, one lifts the tracing joint from the plane, reassembles the linkage, except for the pinned joints, into a new \emph{equivalent} planar position, and continues to trace, repeating the process until nothing new can be drawn. Abbott has pointed out that corrected Kempe proofs  extend naturally to cover this context also, and so one may obtain the fact, due to Kapovich and Millson, that any compact real algebraic curve (ie., a compact real variety $V(f)$) is realisable as the \emph{orbit} of a single joint of a pinned planar linkage. This is a corollary of the next theorem.

Some care must be taken when comparing different variants of Kempe universality since joints may be said to  \emph{describe} $C$,  \emph{trace} $C$,  have \emph{orbit} or \emph{workspace} equal to $C$, and so forth. 
Abbott uses the terminology \emph{drawable} to describe a set which coincides with the orbit of the joint
and he obtained the following theorem by
using both braced contraparallelograms and parallelograms.

\begin{thm}\label{t:kapmill}
Let $f \in \bR[x, y]$ be a polynomial, and
let $D$ be a closed disk in the plane. Then there exists a planar linkage that draws the set
$D \cap V(f)$.
\end{thm}

A key point for a Kempian proof of this  theorem is that not only must the bifurcation motions of parallelograms and contraparallelograms be ruled out but that the configuration space of the resulting pinned framework $(G,p)$ should similarly present no unintended bifurcations in the motion of the special joints $p_k$. Thus  the joints $p_3, p_4, \dots $ should be uniquely determined by the pair $\theta, \phi$ even if the configurations of the connecting building blocks are not uniquely determined. 
Note for example that the (braced) separation translator with input pair $\{A, B\}$ and output pair $\{A', B'\}$ has the property that $B'$ is determined by
$A, B$ and $A'$ even though the translator for this admits two configurations.

We also remark that there is a related diffeomorphic theory considered by  Kapovich and Millson \label{t:kapMillDifeo} which involves controlling the full configuration space of a linkage. For example it can be shown that any compact connected smooth manifold is diffeomorphic to a connected component of the configuration space of a pinned planar linkage.
Related developments in this direction may be found in Kourganoff \cite{kou}.

\subsection{Correcting Kempe}\label{ss:kempe}
 {How wrong was Kempe ?  His  "reversor" linkage, which performs angle doubling, is a pin-bar structure as indicated in Figure \ref{f:reversor}, with $2$ contraparallelogram substructures. The "multiplicator" linkage on the other hand performs a general angle addition $\alpha+\beta$. Concatenating these he obtains linkages that, crucially, can effect the multiplication $\theta \to r\theta$. The multiplicators themselves he obtains by superimposing two reversors by merging their reflection bar $OX$, as in Figure \ref{f:additor}.

\begin{center}
\begin{figure}[ht]
\centering
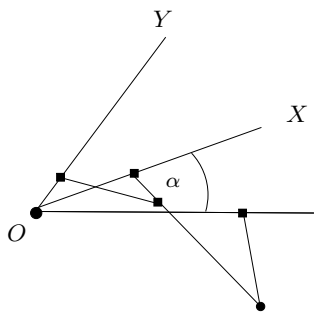
\caption{Kempe's reversor: the output bar $OY$ has angle $\alpha + \alpha$.}
\label{f:reversor}
\end{figure}
\end{center}

\begin{center}
\begin{figure}[ht]
\centering
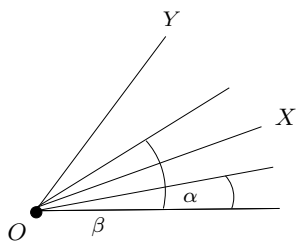
\caption{Kempe's multiplicator: the output bar $OY$ has angle $\alpha + \beta$.}
\label{f:additor}
\end{figure}
\end{center}

For the reversor there is no bifurcation of motion if $0<\alpha <\pi/4$, since the contraparallelograms remain nondegenerate in this range (and more). For the multiplicator there is degeneracy if $\beta = \alpha$ but there is no bifurcation if
$0<\alpha <\beta < \pi/4$. 

With these observations we note that by moving the origin a rather considerable distance to the left we may arrange that the lengths $a$ and $b$ (with $a$ much larger than $b$) are such  that for all pairs $\theta, \phi$ associated with points $(x,y)$ of the disc $D$ we have
\[
0< n\phi < \theta, \quad (n+1)\theta <\pi/4
\]
where $n$ is the degree of $f$.

We can now correct Kempe's argument in a minimal manner by indicating how to adjust the proof of Theorem \ref{t:kempeuniversality} by employing Kempe building block linkages. 

First connect $n-1$ separate additors to the input angle bars in order to have available output angles  $2\theta, \dots  , n\theta$. Also connect further additors to construct the angles
$k\phi$ for $2\leq k\leq n$,
and add reversors to construct $-2\phi, \dots , -n\phi$. Then  connect  additors to construct all the needed positive angles
$r\theta + \sigma_{r,s}s\phi$ with $1 \leq r, s \leq n$. 
In each case no degenerate contraparallelograms occur as $(x,y)$ moves in the disc $D$.
The addition of a fixed angle $\pi/2$ is also straightforward by a rigid connection. 

We may now construct $p_2, p_3, \dots , p_C$ as before by using a sequence of appropriate separation translators. 

It follows that Kempe's arguments have been closer
than  previously thought to showing that the curves describable by linkwork include all (connected) bounded curves of the $n^{th}$ degree. One may avoid the contraparallelogram degeneracies by moving the origin, as above, while the parallelogram degeneracies may be  corrected by the simple prismatic bracings indicated in Figure \ref{f:copierC}.}
\bigskip

\noindent \emph{Acknowledgement.} The results here are part of the EPSRC project \emph{Infinite bond-node frameworks}, EP/P01108X/1.



\end{document}